\documentclass{article}
\usepackage{graphicx} 
\usepackage{amsmath, amsthm, amssymb, amsfonts}
\usepackage{thmtools}
\usepackage{setspace}
\usepackage[top=40mm, bottom=40mm, inner=25mm, outer=25mm, headsep=10mm, footskip=12mm]{geometry}
\usepackage{float}
\usepackage{array}
\usepackage[utf8]{inputenc}
\usepackage{wrapfig}
\usepackage[dvipsnames]{xcolor}

\newcommand{\R}{\mathbb{R}}
\newcommand{\Z}{\mathbb{Z}}
\newcommand{\N}{\mathbb{N}}
\newcommand{\e}{\varepsilon}
\newcommand{\abs}[1]{\left|#1\right|}
\newcommand\blfootnote[1]{%
  \begingroup
  \renewcommand\thefootnote{}\footnote{#1}%
  \addtocounter{footnote}{-1}%
  \endgroup
}

\declaretheoremstyle[name=Theorem,]{thmsty}
\declaretheorem[style=thmsty,numberwithin=section]{theorem}

\declaretheoremstyle[name=Corollary,]{thmsty}

\declaretheoremstyle[name=Proposition,]{prosty}
\declaretheorem[style=prosty,numberlike=theorem]{Proposition}

\declaretheoremstyle[name=Lemma,]{prcpsty}
\declaretheorem[style=prcpsty,numberlike=theorem]{Lemma}

\declaretheoremstyle[name=Definition,]{prcpsty}
\declaretheorem[style=prcpsty,numberlike=theorem]{defn}

\declaretheoremstyle[name=Remark,]{prcpsty}
\declaretheorem[style=prcpsty,numberlike=theorem]{Remark}
\begin{document}
\author{Andrea Braides, Edoardo Voglino and Matteo Zanardini \\
\small SISSA, via Bonomea 265, Trieste, Italy,  {\tt \{abraides,evoglino,mzanardi\}@sissa.it}
}
\title{Microstructures and anti-phase boundaries \\
 in long-range lattice systems}
\date{}

\maketitle
\begin{abstract}
      We study the effect of long-range interactions in non-convex one-dimensional lattice systems in the simplified yet meaningful assumption that the relevant long-range interactions are between $M$-neighbours for some $M\ge 2$ and are convex. If short-range interactions are non-convex we then have a competition between short-range oscillations and long-range ordering.  In the case of a double-well nearest-neighbour potential, thanks to a recent result by Braides, Causin, Solci and Truskinovsky, we are able to show that such a competition generates $M$-periodic minimizers whose arrangements are driven by an interfacial energy. Given $M$, the shape of such minimizers is universal, and independent of the details of the energies, but the number and shapes of such minimizers increases as $M$ diverges.   


    {\bf MSC Class (2020):} 39A12, 34K31, 82B20, 49J45, 46E39.
\end{abstract}

\blfootnote{Preprint SISSA  08/2024/MATE}

\section{Introduction}
In this paper we study the behaviour of boundary-value minimum problems related to one-dimensional long-range lattice energies, which in the greatest generality can be stated as the asymptotic behaviour as $n\to+\infty$ of solutions $u^n=\{u^n_i\}_i$ of the minimization of functionals of the form 
\begin{equation}\label{eq=0-0}
E_n(u)=\sum_{k=1}^n\sum_{i=k}^n \psi_k(u_i-u_{i-k})
\end{equation}
on $n+1$-tuples $u=\{u_i\}_i$ of parameters with $u_i\in\mathbb R$, subjected to boundary conditions $u_0=0$  and $u_n=n\ell$. In this generality, the problem is very complex and leads to a variety of different issues with competing short and long-range oscillations and concentration effect, except for the trivial case when all $\psi_k$ are convex, for which the minimizer is essentially unique and coincides with the linear function $u_i= i\ell$, except for boundary effects, which are asymptotically negligible upon some technical assumptions on $\psi_k$. Nevertheless, an averaged description of minimizers is possible in the spirit of  $\Gamma$-convergence. To this end, we regard the energies $E_n$ as defined on functions $u\colon[0,1]\to\mathbb R$, with domain the piecewise-affine functions defined by $u(x)= \frac1n u_{\lfloor x/n\rfloor}$ for some $(n+1)$-tuple $\{u_i\}_i$, in such a way that the derivative of $u$ on $((i-1)/n, i/n)$ is $u_i-u_{i-1}$. The $\Gamma$-convergence of such energies can then be studied with respect to the $L^1$-convergence. Upon some growth hypotheses on $\psi_k$ that ensure that limits of sequences $u_n$ with energy of order $n$ belong to some $W^{1,p}(0,1)$, the $\Gamma$-limit of $\frac1n E_n$ can be written as
\begin{equation}\label{eq=1-0}
F(u)=\int_0^1 \psi_{\rm hom} (u')dt,
\end{equation}
for a convex function $\psi_{\rm hom}$ resulting from a non-linear homogenization process (see \cite{BGS}, and \cite{AC} for the higher-dimensional case).  

The $\Gamma$-convergence above only ensures that the (interpolations of the) minimizers of boundary-value problems for $E_n$ converge to the corresponding minimizer, or to one of the minimizers, of the analog continuum boundary-value problem related to $F$, but brings no further information on their behaviour, which may depend on $\ell$.  Note that if $\psi_{\rm hom}$ is strictly convex at  $\ell$ then the unique minimizer is the linear function $u(x)=\ell x$, while at non-strictly convex points discrete solutions may converge to  a particular choice among minimizers.

A particular class of energies for which an analysis of $\psi_{\rm hom}$ leading to a description of the convergence of discrete minimizers has been possible is the one studied in \cite{BCST}, where the non-convexity is confined to nearest-neighbour interactions governed by $\psi=\psi_1$, and the long-range potentials are quadratic; that is, $\psi_k(z)= a_k z^2$, with $a_k$ non-negative, and the energies can be written as 
\begin{equation}\label{eq=1}
E_n(u)=\sum_{i=1}^n \psi(u_i-u_{i-1})+\sum_{k=2}^n\sum_{i=k}^n a_k(u_i-u_{i-k})^2
\end{equation}
In that case, the properties of minimizers can be linked to properties of the sequence $a_k$. In particular, in \cite{BCST} the case of double-well $\psi$ is studied, for which a prototype is
\begin{equation}\label{eq=0}
\psi(z)=\min \{(z-1)^2,(z+1)^2\},
\end{equation}
in which case it is possible to describe the patterns of the minimizers by tracing whether the 
value $z_i=u_i- u_{i-1}$ lies in one ``well'' (i.e., $z_i\le0$) or the other one (i.e., $z_i\ge0$).
As such patterns of minimizers are concerned, we recall the  following interesting  characterization of minimizers of problems \eqref{eq=1} in the case when $a_k=0$ for all $k\ge 2$ except for one value $k=M$: either 

\begin{itemize}
\item[(a)] minimizers $u_i$ are such that $z_i=u_i- u_{i-1}$ tend (for $n$ large) to be $M$-periodic with average $z$, and take only two values, one, for $m$ indices in the period, in one well and the second one, for $M-m$ indices in the period, in the second well, or 

\item[(b)] $z_i$ defined as above tend (for $n$ large) to be a mixture of two periodic functions as in (1) with some $z'$ and $z''$ in the place of $z$ and for two consecutive values $m$ and $m+1$ between $0$ and $M$. 
\end{itemize}

\noindent This characterization extends a formula known before when only $a_2\neq 0$ (see \cite{BGS}), in which case we have the only three possibilities that either 
the parameters $z_i$ tend to take a constant value in the first or in the second well, or that we have a $2$-periodic pattern mixing values in both wells.

The appearance of microstructure is a recurring feature of non-convex variational systems. Such microstructures may be driven by a scale-free relaxation phenomenon described by convexification or quasiconvexification of the original energy densities (see e.g.~the books by Buttazzo \cite{Buttazzo} or Dacorogna \cite{Dacorogna}), or present more regular patterns at a specific scale due to competing long-range and short-range effects (as in the seminal paper by S.~M\"uller \cite{SM}; see also \cite{AM}). Minimizers of $E_n$ are similar to the latter, with oscillations both driven by short- and long-range microscopic interactions, and by mesoscopic non-convexity. 

Simple examples of variational problems exhibiting microscopic oscillations are lattice systems defined on ``spin functions''; i.e., functions taking only a finite number of values, the traditional choice being $-1$ and $1$. If the energies are ``frustrated''; that is, the system presents interaction potentials that cannot be all separately minimized at the same time by a single function, then minimization may produce periodically modulated phases (see \cite[Chapter 7]{ABCS}). Often, the determination of the period and shape of such minimizers is a non-trivial matter as in the case of infinite-range antiferromagnetic systems studied by Giuliani et al.~\cite{GLL}, and has interesting continuum counterparts (see e.g.~\cite{Daneri-Runa}). 

In the case of  double-well problems, the location of the parameter in one or the other well relaxes the strict constraint that the parameter takes two values; that is, the constraint $z\in\{-1,1\}$ is replaced by a potential $\psi(z)$ where $\psi$ is a strictly positive continuous function minimized in $\{-1,1\}$. To distinguish between them, we will call the first type of parameters ``hard spins'' and the second ones ``soft spins''. For the prototypical double-well potential $\psi$ in \eqref{eq=0} it is clear that the two ``wells'' coincide with $z$ negative and $z$ positive. 
If also long-range interactions are taken into account, minimization for soft spins may produce patters analogous to those for frustrated hard spins. Furthermore, the ``soft'' approach allows to include more easily boundary-value problems as above.

In this paper we carry on a fine analysis of the energies $E_n$ in \eqref{eq=1} by examining not only minimizers, but also parameters $u_i$ whose energy in \eqref{eq=1} differs from the minimum by a finite quantity bounded as $n$ tends to $+\infty$. This is done by using a development by $\Gamma$-convergence \cite[Section 1.10]{GCB} and \cite{BT}, and is performed  for double-well $\psi$ and $a_k\neq 0$ only for $k=M$, so that the descriptions (a) and (b) above provide the value of minima. The meaningful definition of convergence for functions $u$ depends on whether we are in case (a) or (b) above. For simplicity of illustration consider that the boundary datum $\ell$ is such that case (a) holds for some $m$. Then, given a sequence $u^n_i$ with bounded energy, the sequence is compact in the following sense: there are a finite number of indices $i^n_j$, which we may suppose to converge after scaling; that is, $i^n_j/n\to x_j$, such that in the intervals in the complement of such indices, each function coincides with a $M$-periodic minimizer $\overline u_i$ as in (a), up to an arbitrary small error. Hence, up to subsequences, each such sequence determines a continuum limit ${\bf u}$ whose derivative $\mathbf{u}'$ takes values in the finite set ${\bf M}_m$ of $M$-periodic minimizers as in (1). A similar argument holds in case (b), for which we can conclude that the continuum limit ${\bf u}$ has derivative with values in ${\bf M}_m\cup{\bf M}_{m+1}$. Once such a piecewise-constant limit is defined, we will prove that the $\Gamma$-limit has the form
\begin{equation}\label{eq=2}
F(u)=\sum_{t\in S({\bf u}')}  \Phi ({\bf u}'(t^-),{\bf u}'(t^+)),
\end{equation}
where $S({\bf u}')$ denotes the discontinuity set of ${\bf u'}$ in $(0,1)$. This shows that minimization may give rise to microscopic patterns ${\bf M}_m$, microscopic incompatibility may give rise to interfaces between elements of ${\bf M}_m$ (anti-phase boundaries) or/and interfaces between elements of ${\bf M}_m$ and ${\bf M}_{m+1}$ (macroscopic interfaces). In order to avoid boundary effects, the analysis will be carried out under some periodicity assumptions.

It is interesting to note that, even though the values of the slopes of microscopic minimizers in ${\bf M}_m$ depends on the average slope, or boundary datum, $\ell$, the set ${\bf M}_m$ has a `universal' form, and its elements are in correspondence with $M$-tuples with $m$ values equal to $1$ and $(M-m)$ values equal to $0$. A final remark is that the presence of $M$-th-neighbour interaction is often compared to that of a singular perturbation with a term containing the $M$-th derivative for continuum double-well problems. However, while in the continuum case the resulting phase-transition energy is essentially independent of $M$ (see \cite{brusca2024higherorder,solci2024freediscontinuity}), in the discrete case our result shows an increasing complexity of minimizers as $M$ increases.

\section{Statement of the result}
We will fix $M\in\mathbb N$ with $M\ge 2$ and functions $\psi_1,\psi_M\colon\mathbb R\to[0,+\infty)$ satisfying the coerciveness condition 
\begin{equation}\label{growth cond hyp} \lim_{\abs{z}\rightarrow \infty} \frac{\psi_1(z)}{\abs{z}} = + \infty.
\end{equation}
We want to study the overall behaviour of functionals with competing nearest-neighbour and M-th-neighbour interactions driven by the potential $\psi_1$ and $\psi_M$, respectively, of the form 
\begin{equation}\label{EnM-inf}
\sum_{i}\psi_1\left({u_{i+1}-u_{i}}\right) + \sum_{i}\psi_M\Big(\frac{u_{i+M}-u_{i}}{M}\Big),
\end{equation}
defined on discrete functions indexed on $\mathbb Z$. In our assumptions the potential $\psi_1$ will be a double-well energy density which favours oscillations through non-convexity, while $\psi_M$ is a convex potential favouring long-range ordering. 

\subsection{Analysis at the bulk scaling}
We preliminary analyze a scaled version, whose analysis will suggest a renormalization argument.
 We use a standard scaling procedure that allows to use an analytic approach by $\Gamma$-convergence, introducing a reference interval $[0,1]$ and the small parameter $\e_n=\frac{1}n$ with $n\in \mathbb N$. The energies above, when we take into account the interaction involved on $n+1$ sites, now parameterized by $\e_ni$ with $i\in\{0,\ldots, n\}$, take the form 
 \begin{equation}\label{EnM-0}
\sum_{i=0}^{n-1} \psi_1\Big(\frac{u_{i+1}-u_{i}}{\e_n}\Big) + \sum_{i=0}^{n-1} \psi_M\Big(\frac{u_{i+M}-u_{i}}{M\e_n}\Big).
\end{equation}
In this notation $u_i=u(\e_n i)$. Note that in the last sum we also take into account the values of $u_i$ for $i\in\{n ,\ldots, n +M-1 \}$.
In the sequel, in order not to have boundary effects, we will define $u_i$ for all values of $i$ using some periodic conditions. 

After this parameterization, we can identify such discrete functions with the piecewise-affine interpolation on $[0,1]$ of the sites $(i\e_n, u_i)$. 
We define the space of such functions
\[
\mathcal{A}_n(0,1) = \{u:[0,1] \rightarrow \R \text{ continuous,  and affine on } (i\e_n, (i+1)\e_n), \,i\in\{0,\dots, n-1\}\},
\]
and the scaled functionals
 \begin{equation}\label{EnM}
 E_{n,M}(u)=\sum_{i=0}^{n-1} \e_n\psi_1\Big(\frac{u_{i+1}-u_{i}}{\e_n}\Big) + \sum_{i=0}^{n-1 }\e_n \psi_M\Big(\frac{u_{i+M}-u_{i}}{M\e_n}\Big).
\end{equation}
Since
$$
\sum_{i=0}^{n-1} \e_n\psi_1\Big(\frac{u_{i+1}-u_{i}}{\e_n}\Big)=\int_0^1 \psi_1(u')\,dt,
$$ 
condition \eqref{growth cond hyp} ensures that functionals $ E_{n,M}$ are equicoercive in $W^{1,1}(0,1)$; namely, that if $u^n$ is bounded in $L^1(0,1)$ and $E_{n,M}(u_n)\le C<+\infty$, then, up to subsequences, $u^n$ converge weakly in $W^{1,1}(0,1)$ and strongly in $L^1(0,1)$. 
The $\Gamma$-limit of $E_{n,M}$ with respect to this convergence is described in the following result, where we also consider periodic conditions. To that end we fix $\ell\in\mathbb R$ and define
    \[W^{1,1}_{\#,\ell}(0,1) = \Big\{u\in W^{1,1}_{\rm loc}(\R): \; u(t)-\ell t \text{ is } L\text{-periodic}\Big\},\]
whose discrete counterpart is 
    \[ \mathcal{A}_{n,\ell}^\#(0,1)= \{u\in W^{1,1}_{\#,\ell}(0,1) : \, {u|}_{[0,1]} \in \mathcal{A}_n(0,1)\}.\]

\begin{theorem}\label{0-conv}
The functionals $E_{n,M}$ $\Gamma$-converge, with respect to the $L^1$-topology, to
the functional defined on $W^{1,1}(0,1)$ by
\begin{equation}
\int_0^1 \psi_0^{**}(u^\prime(t))\;dt,
\end{equation}
   where $\psi_0$ is given by:
\begin{equation}\label{psizero}
    \psi_0(z) = \psi_M(z) + \frac{1}{M}\min\left\{\sum_{k=1}^M\psi_1(z_k): \;  \sum_{k=1}^Mz_k = Mz,\, z_1, \dots z_M \in \R \right\}.
\end{equation}

Furthermore, the convergence is compatible with the addition of periodic condition; that is, 
with fixed $ \ell \in\mathbb R$, the functionals defined by
$$
E^\ell_{n,M}(u) = \left\{\begin{aligned}
        &E_{n.M}(u) &\qquad &u\in \mathcal{A}_{n,\ell}^\#(0,1)\\
        &+\infty &\quad &\text{otherwise}
    \end{aligned}\right.
$$
$\Gamma$-converge to 
$$
    E^\ell_M(u) = \left\{
    \begin{aligned}
        &\int_0^1 \psi_0^{**}(u^\prime(t))\;dt &\qquad &u\in W^{1,1}_{\#,\ell}(0,1) \\
        &+\infty &\quad &\text{otherwise in $L^1(0,1)$.}
    \end{aligned}
    \right.
$$
\end{theorem}

The proof of this result can be found in \cite{BCST} (see also \cite{BGS} for the case $M=2$).

\subsection{Microscopic analysis}
The main point of the analysis at the bulk scaling is the definition of $\psi_0$ which will allow to renormalize energies \eqref{EnM-inf} by subtracting the affine term $r_\ell$ given by the tangent to $\psi^{**}_0$ at $\ell$ and rewriting the sum as
\begin{equation}\label{EnM-inf-r}
\sum_{i =0}^{n-1} \bigg(\psi_M\Big(\frac{u_{i+M}-u_{i}}{M}\Big)+\frac1M\sum_{k=0}^{M-1}\psi_1({u_{i+k+1}-u_{i+k}})- r_\ell\Big(\frac{u_{i+M}-u_{i}}{M}\Big)\bigg).
\end{equation}
This will be formalized as the computation of a higher-order $\Gamma$-limit starting from $E_{n,M}$. 
We will consider periodic boundary conditions. We also make the simplifying assumption that $n$ is a multiple of $M$, the general case requiring a more complex notation being stated explicitly in Section \ref{res-2}, taking into account possible mismatch at the boundary due to incommensurability.
In this case, after noting that $\min E^\ell_M= \psi_0^{**}(\ell)$, and letting $m=m_\ell$ denote the
tangent to $\psi^{**}_0$ at $\ell$, we can consider the energies 
\begin{eqnarray}
\label{defn diff quot Dirichlet}\nonumber
    E^1_{n,M}(u) &=& E^{1,\ell}_{n,M}(u):= \frac{E^\ell_{n,M}(u) - \min E^\ell_M}{\e_n} = \frac{E^\ell_{n,M}(u) - \psi_0^{**}(\ell)}{\e_n}\\ \nonumber
    &=& \sum_{i=0}^{n -1} \bigg(\psi_M\left(\frac{u_{i+M}-u_i}{M\e_n}\right) + \frac{1}{M}\sum_{k=0}^{M-1}\psi_1\left(\frac{u_{i+k+1}-u_{i+k}}{\e_n}\right) - \psi_0^{**}(\ell)\bigg)
    \\ \nonumber
    &=& \sum_{i=0}^{n  {-1}} \bigg(\psi_M\left(\frac{u_{i+M}-u_i}{M\e_n}\right) + \frac{1}{M}\sum_{k=0}^{M-1}\psi_1\left(\frac{u_{i+k+1}-u_{i+k}}{\e_n}\right) - \psi_0^{**}(\ell) -m\Big(\frac{u_{i+M}-u_i}{M\e_n}-\frac{\ell}{\e_n}\Big)\bigg)
    \\ 
    &=& \sum_{i=0}^{n  {-1}}\bigg(\psi_M\left(\frac{u_{i+M}-u_i}{M\e_n}\right) + \frac{1}{M}\sum_{k=0}^{M-1}\psi_1\left(\frac{u_{i+k+1}-u_{i+k}}{\e_n}\right) - r_\ell\Big(\frac{u_{i+M}-u_i}{M\e_n}\Big)\bigg),
\end{eqnarray}
where we have used that $\sum_{i=  {0}}^{ {n-1}} (u_{i+M}-u_i)= M\ell$ thanks to the $n$-periodicity of $u_i -\ell\e_n i$.

Until now we have made no assumptions on $\psi_1$ and $\psi_M$.  We study a particular case in which $\psi_M$ is a strictly convex function and $\psi_1$ is a double-well potential of the form
\[\psi_1(z) = \min \{W_1(z), W_2(z)\},\]
where $W_1$ and $W_2$ are two smooth convex functions. Note that this is not a very restrictive hypothesis since in the determination of the $\Gamma$-limit only the values of $W_1$ and $W_2$ close to the bottom of the wells will be taken into account, so that more general $\psi_1$ of double-well type can be included in this analysis.
We also make the assumption that $W_1$ and $W_2$ satisfy
\[W_i(z) \geq cz^2 -\frac1c,\]
for some $c>0$. 
This structure provides a useful representation of $\psi_0$. Indeed, we can distinguish two sets $A_1 = \{x\in \R \, |\, \psi_1(x) = W_1(x)\}$ and $A_2 = \{x \in \R \, |\, \psi_1(x) = W_2(x)\}$. Then, we can rewrite the minimum problem in $\psi_0$ as follows:
\begin{eqnarray}\label{interior-min}\nonumber
&&\min\Bigg\{\sum_{k=1}^M\psi_1(z_k) :\; \sum_{k=1}^Mz_k = Mz\Bigg\} \\
&=& \nonumber\min_{0\leq j\leq M}\, \min\Bigg\{\sum_{k=1}^j W_1(z_k) + \sum_{k=j+1}^M W_2(z_k):\; \sum_{k=1}^Mz_k = Mz,\, z_1,\dots, z_j \in A_1, \, z_{j+1},\dots, z_M \in A_2\Bigg\}\\
&=& \min_{0\leq j \leq M}\, \min\Big\{j W_1(z_1) + (M-j) W_2(z_2):\; jz_1 + (M-j)z_2 = Mz,\, z_1\in A_1, \, z_2\in A_2\Big\},
\end{eqnarray}
where the last equality follows from Jensen inequality. 

\begin{figure}[h]
\centering
\includegraphics[width=0.4\linewidth]{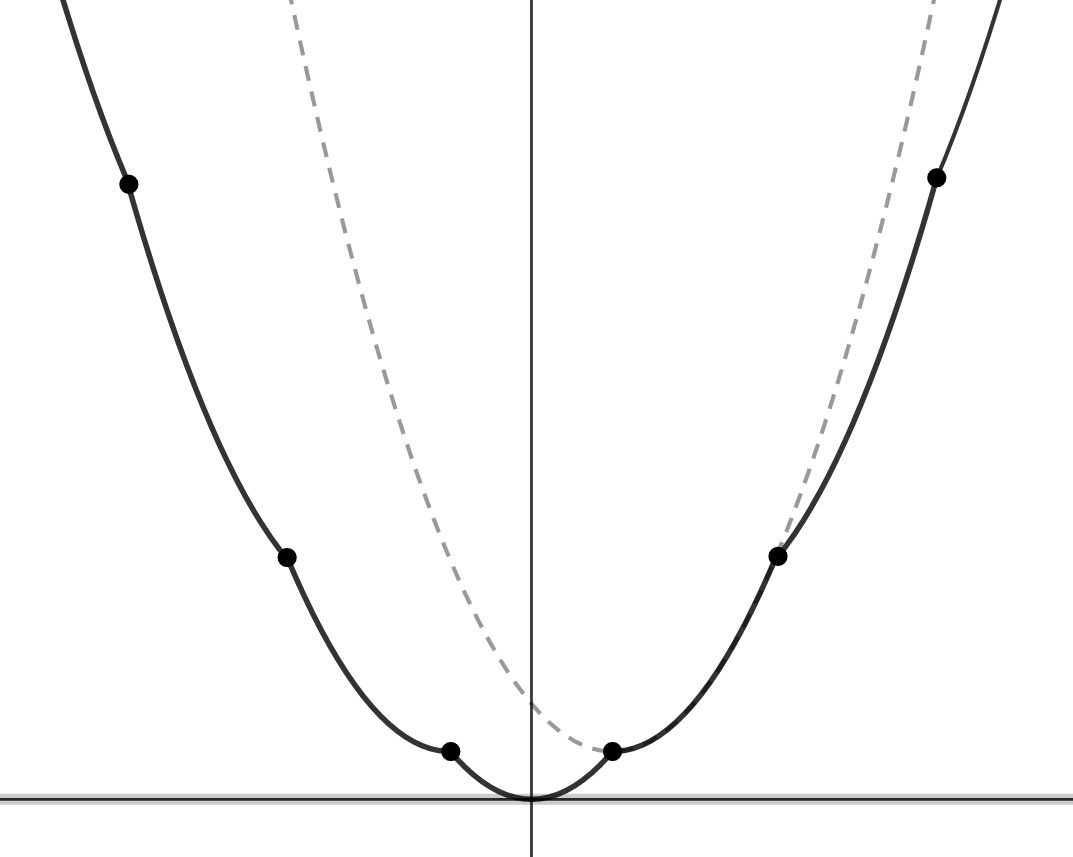} 
\label{fig:psi0} 
\caption{representation of $\psi_0$.}\label{Fig1}
\end{figure}
We can define $\overline{\psi}_j(z)=\psi_M(z) + \frac{1}{M}f_j(z)$, where $f_j(z)$ is the value of the inner minimum problem in \eqref{interior-min}.
 In this way $\psi_0$ can be represented as follows:
\[\psi_0(z) = \min_{0\leq j \leq M} \overline{\psi}_j(z).\] See Fig.~\ref{Fig1} for a typical form of $\psi_0$. 

A non-trivial result in \cite{BCST} (see Theorem 4.1 therein) shows that $\psi_0^{**}$ will alternate non-degenerate intervals $K_j = (z^l_j, z^r_j)$ in which $\psi_0^{**}(z) = \psi_0(z) = \overline{\psi}_j(z)$ and intervals $J_j = [z^r_{j-1}, z^l_j]$, for $j\geq 1$, in which $\psi_0^{**}(z) = r_j(z)$ is a straight line. 
In this notation $z^l_0 = -\infty$ and $z^r_M = + \infty$.

Referring to \eqref{defn diff quot Dirichlet}, this description gives that if $\ell\in K_j$ then 
$$
r_\ell(z)= \psi_0(\ell)+  \psi'_0(\ell)(z-\ell)=  \overline{\psi}_j(\ell)+  \overline{\psi}_j'(\ell)(z-\ell),
$$
while if $\ell\in J_j$ then $r_\ell$ is the affine function whose graph passes through $(z^r_{j-1}, \overline{\psi}_{j-1}(z^r_{j-1}))$
and $(z^l_{j}, \overline{\psi}_{j}(z^l_{j}))$.

Note that if we consider
\begin{equation}
\label{mathcal E}
    \mathcal{E}_n^i(u) = \psi_M\left(\frac{u_{i+M}-u_i}{M\e_n}\right) + \frac{1}{M}\sum_{k=0}^{M-1}\psi_1\left(\frac{u_{i+k+1}-u_{i+k}}{\e_n}\right) - r_\ell\left(\frac{u_{i+M}-u_i}{M\e_n}\right),
\end{equation}
then these values are still positive since $\mathcal{E}^i_n(u) \geq \psi_0^{**}\left(\frac{u_{i+M}-u_i}{M\e_n}\right) - r_\ell\left(\frac{u_{i+M}-u_i}{M\e_n}\right)\geq 0$. 

Before stating the convergence result we need some definitions. 
\begin{defn}
\label{defn min set}
    We define $\mathbf{M}^\alpha\subset \R^M$ as the set of minimizers for the problem
    \begin{equation}
    \label{P}
    \tag{P}
       \min \left\{\sum_{i=1}^M \psi_1(z_i) : \;  \sum_{i=1}^M z_i = M \alpha \right\}. 
    \end{equation}
    We will then set    \[
    \mathbf{M}_\alpha = \begin{cases}
        \mathbf{M}^\alpha & \text{ if } \alpha\in \text{int}(\{z\in \text{dom}\,\psi_0 \,|\, \psi_0(z) = \psi_0^{**}(z)\}) \\
        \bigcup_{z\in J_j\cap \{\psi_0(z) = \psi_0^{**}(z)\}} \mathbf{M}^{z} 
        & \text{ if } \alpha \in J_j.
    \end{cases}
    \]
\end{defn}

\begin{Remark}
    $\mathbf{M}^\alpha$, and consequently also $\mathbf{M}_{\alpha}$, is closed under permutation, that is, if $z = (z_1,\dots, z_M) \in \mathbf{M}^\alpha$, then for any permutation $\sigma$ also $(z_{\sigma(1)}, \dots, z_{\sigma(M)})$ belongs to $\mathbf{M}^\alpha$.
\end{Remark}
\begin{Remark}
Using the notation for $K_j$ and $J_j$, Definition \ref{defn min set} can be simplified as
\[\mathbf{M}_\alpha = \begin{cases}
        \mathbf{M}^\alpha & \text{ if } \alpha\in \bigcup_{j=0}^M K_j \\
        \mathbf{M}^{z^r_{j-1}}\cup \mathbf{M}^{z^l_j}
        & \text{ if } \alpha \in J_j,\  j\in\{1,\ldots,M\}.
    \end{cases}
\]
    Since $\psi_1(z) = \min \{W_1(z), W_2(z)\}$ then we can study the cardinality of the sets $\mathbf{M^{\alpha}} $ because, for a fixed $j \in \{ 0, \dots M \}$, the minimum problem
    \[
    \min\Big\{j W_1(z_1) + (M-j) W_2(z_2) : \; jz_1 + (M-j)z_2 = M \alpha,\, z_1\in A_1, \, z_2\in A_2\Big\} \qquad \qquad (P_j)
    \]
    admits a unique solution by strict convexity. Each minimizer $(z_1,z_2)$ of $(P_j)$ corresponds to $\binom{M}{j}$ minimizers for \eqref{P}; that is, the number of possible $M$-tuple such that $j$ entries are $z_1$ and $M-j$ are $z_2$. 
    This shows that indeed for each $\alpha\in [z^l_j, z^r_j]$ the set ${\bf M}^\alpha$ can be identified with a set ${\bf M}_j$ depending only on $j$, and with this notation ${\bf M}_\alpha={\bf M}_{j-1}\cup {\bf M}_j $
if $\alpha\in J_j$.

\end{Remark}

In order to study the limit behaviour of $E_{n,M}^1$, we need a notion of convergence of discrete $u^n$ to a vectorial function on the continuum.

\begin{defn}
    Given $u\in \mathcal{A}_{n,\ell}^\#(0,L)$, we consider the subintervals
    \[
    I_j = [(j-1)\e_n, j \e_n] \qquad \text{for } j \in \Z
    \]
    and for $i \in \Z$ we define the set $A_i= \bigcup_{k=1}^MI_{iM+k}$, which is
    the union of $M$ such consecutive intervals.
    For $k\in\{1,\dots, M\}$, we define $u_k$ the $k$-th {\em M-interpolation} of $u$ as the piecewise-affine function obtained extending the slope $z^{Mi+k}$, where
     \[
    z^j := \frac{u_j-u_{j-1}}{\e_n},
    \]
to the whole $A_i$; that is, $u_k$ is defined by
    \[\begin{cases}
        u_{k,0} = u_0 , \\
        u'_k(t)=z^{Mi+k}\quad & \text{if } t\in A_i\\
    \end{cases}
    \]
\end{defn}
\begin{defn}
\label{def conv}
    We say that a sequence of discrete functions $(u^n)_n$ {\em converges} to $\overline{\mathbf{u}} = (\overline{u}_1, \dots, \overline{u}_M)$ in a functional topology (for instance $L^1(0,1)$ or $L^{\infty}(0,1)$) if for each $k$ the $k$-th interpolation $u_k^n$  converges to $\overline{u}_k$ in that topology.
\end{defn}

\begin{theorem}[Equi-coerciveness of the energies]
\label{compattezza}
If a sequence $(u^n)_{n\in \N}$ satisfies $\sup_n E^1_{n,M}(u^n) < + \infty$, then, up to addition of a constant and extraction of a subsequence, the sequence converges uniformly to some $M$-tuple of piecewise-affine functions $\mathbf{\overline{u}} = (\overline{u}_1, \dots, \overline{u}_M)$
with  $\overline{\mathbf{u}}'(t)\in \mathbf{M}_{\ell}$ for almost every $t$.
Moreover,  there exists  a finite set $S\subset (0,1)$ such that $u^n$ converges in $W^{1,\infty}_{\rm loc}((0,1) \setminus S)$ to $\mathbf{\overline{u}}$ and the jump set $S(\overline{\mathbf{u}}')$
is contained in $S$.
\end{theorem}

This compactness theorem justifies the use of the convergence in Definition \ref{def conv} for the  $\Gamma$-limit of $  E^1_{n,M}$. In order to describe it we define the following transition energy.

\begin{defn} \label{def-trans}
    Let $\ell \in \R$ and $\mathbf{z} = (z_1, z_2, \dots z_M), \mathbf{z'}= (z'_1, z'_2, \dots z'_M) \in \mathbf{M}_\ell$. The {\em transition energy} between $\mathbf{z}$ and $\mathbf{z'}$ is defined by
    \begin{align} \label{formula magica}
        \Phi ^{(\ell)} (\mathbf{z}, \mathbf{z'}) : = \inf_{N \in \N} \min_u \bigg \{ \notag & \sum_{i  \in \Z} \Big( \psi_M \Big( \frac{u_{i+M}-u_i}{M} \Big) 
         + \frac{1}{M} \sum_{k=0}^{M-1} \psi_1 (u_{i+k+1}-u_{i+k}) - r_\ell\Big( \dfrac{u_{i+M}-u_i}{M} \Big) \Big): \notag \\
        & \; \; u : \Z \rightarrow \R,\,  u_i= u_{\mathbf{z}}(i) \text{ for } i \leq -N,\,
         u_i= u_{\mathbf{z'}}(i) \text{ for } i \geq N \Big\},
    \end{align}
    where $r_\ell$ is the tangent line to $\psi_0^{**}$ at $\ell$ computed at the point $x$ and,
    for any $\mathbf{z} \in \R^M$, $u_\mathbf{z}$ is the piecewise-affine function $\R\longrightarrow \R$ defined as follows:
    \begin{equation}
    \label{defn u_z}
    \begin{cases}
        u'_{\mathbf{z}}(t) = z_k \quad \text{for } t\in (k-1, k) + M\Z\\
        u_\mathbf{z}(0) = 0
    \end{cases}
    \end{equation}
    \end{defn}
    
    With this definition we can state the $\Gamma$-convergence result as follows
    
    \begin{theorem}[first-order $\Gamma$-limit with periodic boundary conditions]\label{1-teo}
    Assume that $\psi_M$ is a $C^1$ strictly convex function, $\psi_1$ is of class $C^1$ in a neighbourhood of $\mathbf{M}_\ell$ and $\psi_0^{**}$ is differentiable in $\ell$. We define the domain
    \begin{align*}
        D_\ell^\# = \{\mathbf{u}: \R\rightarrow \R^M \,|\,& \mathbf{u} \text{ (locally) piecewise affine}, \sum_{k=1}^{M} u_k(t)- M\ell t  \text{ 1-periodic,}     \\ & \mathbf{u}'\text{ 1-periodic and } \mathbf{u}'(t)\in \mathbf{M}_{\ell}\; \text{ for all } t\in \R\setminus S(\mathbf{u}')\}.
    \end{align*}
    Then
    $E^{1}_{n,M}$ $\Gamma$-converges, with respect to the $L^{\infty}$ topology, to 
\begin{equation}\label{Phifu}
E^1_M (\mathbf{u}) =  \sum_{t\in S(\mathbf{u}')\cap (0,1]}\Phi^{(\ell)}(\mathbf{u}'(t-), \mathbf{u}'(t+)),
\end{equation}
    with domain $D_\ell^\#$.
\end{theorem}

\section{Proof of the results}
In this section, for a greater generality, we consider $L>0$ and functions parameterized on $[0,L]$ instead of $[0,1]$.
In this case $\e_n=\frac{L}n$, and we extend all the notation introduced in the case $L=1$. With fixed $\ell\in\mathbb R$, the energies we consider are directly written in the form
\begin{eqnarray}
\label{defn diff quot Dirichlet-1}
    E^1_{n,M}(u) := \sum_{i=0}^{n  {-1}} \bigg(\psi_M\left(\frac{u_{i+M}-u_i}{M\e_n}\right) + \frac{1}{M}\sum_{k=0}^{M-1}\psi_1\left(\frac{u_{i+k+1}-u_{i+k}}{\e_n}\right) - r_\ell\Big(\frac{u_{i+M}-u_i}{M\e_n}\Big)\bigg),
\end{eqnarray}
defined on $\mathcal{A}^\#_n(0,L)$.
Note that $ E^1_{n,M}(u)= \sum_{i= {0}}^{ {n-1}}\mathcal{E}^i_n(u)$, with $\mathcal{E}^i_n$ be given by \eqref{mathcal E}.

\subsection{Compactness}\label{res-1}
We can state the compactness result as in Theorem \ref{compattezza} independently of the boundary condition as follows.

\begin{Proposition}\label{prop-comp}Let $\mathcal{E}^i_n$ be given by \eqref{mathcal E}. If 
$$
\sup_n \sum_{i= {0}}^{ {n-1}}\mathcal{E}^i_n(u^n)=: C<+\infty,
$$
then $u^n$ satisfies the claim of Theorem \ref{compattezza} with the interval $[0,L]$ in the place of the interval $[0,1]$.
\end{Proposition}

\begin{proof}
    Let $\eta>0$, then we define the set
    \[
    I_n (\eta ) : = \{ i \in \{0,1, \dots n-1\} \, : \, \mathcal{E}^i_n (u^n) > \eta \},
    \]
    and, as a consequence of the bound on the sum, we get that
    \[
    \sup_n \# I_n (\eta) \leq C(\eta) \sim C/\eta <+ \infty.
    \]
    Therefore, if $i\notin I_n(\eta)$, by adding and subtracting $\psi_0\left(\frac{u^n_{i+M}-u^n_i}{M\e_n}\right)$ to $\mathcal{E}^i_n(u^n)$, we obtain the two inequalities
    \[
        0 \leq \psi_M\left(\frac{u_{i+M}^n-u_i^n}{M\e_n}\right) + \frac{1}{M}\left(\sum_{k=0}^{M-1}\psi_1\left(\frac{u^n_{i+k+1}-u^n_{i+k}}{\e_n}\right)\right) - \psi_0\left(\frac{u^n_{i+M}-u^n_i}{M\e_n}\right)\leq \eta
    \]
    \[
        0 \leq \psi_0\left(\frac{u^n_{i+M}-u^n_i}{M\e_n}\right) - r_\ell\left(\frac{u^n_{i+M}-u_i^n}{M\e_n}\right)\leq \eta.
    \]
Note that, if $z$ and $(z_1, \dots, z_M)$ satisfy
    \begin{enumerate}
        \item[a)] $\psi_M\left(z\right) + \frac{1}{M}\left(\sum_{k=0}^{M-1}\psi_1\left(z_k\right)\right) - \psi_0\left(z\right)\leq \eta, \quad \text{for }\sum_{k=0}^{M-1} z_k = M z$
        \item[b)] $\psi_0\left(z\right) - r\left(z\right)\leq \eta$
    \end{enumerate}
    then $(z_1,\dots,z_M)$ is close to a minimizing $M$-tuple for the min in the definition of $\psi_0$, while $\psi_0(z)$ is close to $r(z)$, the tangent line of $\psi_0^{**}$ in $\ell$. This means that if $\psi_0(\ell)=\psi_0^{**}(\ell)$ then $z$ is close to $\ell$, while if $\ell$ is in some $J_j$ then $z$ is close to either $z^r_{j-1}$ or $z^l_J$.
    Hence, for $\eta$ small enough we can find $\e>0$ so that if (a) and (b) are satisfied then
    \[ \text{dist}\left((z_1,\dots, z_M), \mathbf{M}_\ell\right)\leq \e < \frac{1}{2}\min \{\abs{\mathbf{z}'-\mathbf{z}''}\,:\, \mathbf{z}',\mathbf{z}''\in \mathbf{M}_\ell\}.
    \]
    This entails that for each $i\notin I_n(\eta)$ there exists a unique $\mathbf{z}^n_i = (z^n_{i,1}, \dots, z^n_{i,M})\in \mathbf{M}_\ell$ such that 
    \begin{equation}
    \label{slope almost min}
\abs{\left(\frac{u^n_{i+1}-u^n_{i}}{\e_n}, \dots, \frac{u^n_{i+M}-u^n_{i+M-1}}{\e_n}\right) - \mathbf{z}^n_i}\leq \e,
        \implies\;\abs{\frac{u^n_{i+k}-u^n_{i+k-1}}{\e_n} - z^n_{i,k}}\leq \e, \quad 
    \end{equation}
for all $ k\in\{1, \dots, M\}$.  
 Since $I_n(\eta)$ is a finite set, we can identify $N_n$ pairs of indices $0= j_0 \leq i_1<j_1< i_2 < j_2 <\dots < i_{N_n}< j_{N_n} \leq i_{N_n+1}= n$ such that the slopes of the $M$-interpolations of $u^n$ are close to an element of $[\mathbf{z}_{i_k}^n]$ in the sense of (\ref{slope almost min}) for all indices between $i_k$ and $j_k$, and they are not otherwise.
 
Since $C \geq E^1_{n,M}(u^n) \geq \eta (N_n-1)$, we must have that $N_n$ is bounded with respect to $n$. 
    Then, we can assume, up to extracting subsequences, that $N_n$ is constantly equal to $N$. Similarly, since $\mathbf{M}_\ell$ is a finite set, we can assume $\mathbf{z}_{i_k}^n = \mathbf{z}_k$.
With fixed $k\in\{ 2, \dots, N\}$, for each $n$ we chose an index $\Tilde{i}_k \in \{j_{k-1}+1, \dots, i_{k}-1\}$ and define the sequence $\{\Tilde{x}^n_k\}$ such that $\Tilde{x}^n_k = \Tilde{i}_k\e_n$, then, up to subsequences, we can assume $x^n_k \rightarrow x_k$.
    However, $\abs{i_k- j_{k-1}}$ must be bounded independently of $n$ and $k$, since
    \[C \geq E^1_{n,M}(u^n) \geq \sum_{k = 1}^{N+1}\eta \abs{i_k- j_{k-1}}.\]
    Then we have $\abs{\Tilde{x}^n_k - i\e_n}\leq \frac{C}{\eta}\e_n$ for any $i \in \{j_{k-1}+1, \dots, i_{k}-1\}$. Therefore, the whole $\{j_{k-1}+1, \dots, i_{k}-1\}\e_n$ is converging to $x_k$.
    
    Finally, we define the sets $S = \bigcup_{k=1}^N \{x_k\}$ and, for a fixed small $\delta$, $S_\delta = \bigcup_{k=1}^N (x_k-\delta, x_k + \delta)$.  
By \eqref{slope almost min}, we have that for $n$ large enough the $s$-th $M$-interpolation satisfies
    \begin{equation}
    \label{equicont}
        \sup_{t\in (0,1)\setminus S_\delta} \abs{(u_{s}^n)'(t) - Z_s(t)} \leq \e,
    \end{equation}
    where $\mathbf{Z}(t) = \left(Z_1(t), \dots, Z_M(t)\right)$ is a piecewise-constant function such that $\mathbf{Z}(t) = \mathbf{z}_k$ for $t \in (x_{k-1}, x_k)$.
    
    Note that from the equicoerciveness of $E_{n,M}$ and hence also of $E^1_{n,M}$, we obtain that $u_{n,s}$ are a precompact sequence in $H^1(0,L)$, so that we can suppose that they converge uniformly. 
    But from \eqref{equicont}, the subsequence is such that $(u_{s}^n)'$ is converging to $Z_s$. 
    In conclusion, by applying a diagonal argument, we have proved that, up to subsequences, $u^n$ converges in $W^{1,\infty}_{\rm loc}((0,L)\setminus S)$ (in the sense of Definition \ref{def conv}) to a function $\mathbf{u}$ such that $\mathbf{u}'(t) = Z(t) \in \mathbf{M}_\ell$ and the jump set $S(\mathbf{u}') \subseteq S$. \end{proof}
    
    Note that in this proof we have not used the periodicity condition. If it is taken into account that we have to note that $S(\mathbf{u}')$ can also contain the point $0$.

\subsection{Computation of the Gamma-limit}\label{res-2}
We now compute the $\Gamma$-limit subjected to periodic boundary conditions, without the simplifying assumption that $n$ be a multiple of $M$ used for presentation purposes in the previous section. We will show that the limit exists and can be characterized if more in general $n=q$ modulo $M$. The energy will have the same form as in the case $q=0$, but the characterization of the domain of the $\Gamma$-limit will depend on $q$, since $M$-interpolations do not necessarily inherit the periodicity condition of $u$. Note in particular that the limit as $n\to+\infty$ does not exist.

\smallskip
For  $n\in\mathbb N$ we let $q \in \{0,\dots, M-1\}$ denote its class modulus $M$ ($n\equiv q$ mod $M$), then given $u$ a piecewise-affine function in $W^{1,1}_{\#, \ell}(0,L)$, for any index $Mi+q$ with $q \in \{0,1,\dots,M-1\}$, its $M$-interpolations $u_k$ will satisfy:
\[\frac{u_{k,Mi+q} - u_{k,Mi + q-1}}{\e_n} = \frac{u_{Mi+k} - u_{Mi+k-1}}{\e_n} = \frac{u_{n + Mi+k} - u_{n+Mi+k-1}}{\e_n}.\]
Since $n = Mr+ q$, we write
\[\frac{u_{k,Mi+k} - u_{k,Mi + k-1}}{\e_n} = \frac{u_{M(i+r)+q+k} - u_{M(i+r) + q + k-1}}{\e_n},\]
which is linked to the $q+k$ interpolation (mod $M$) shifted of $\left\lfloor \frac{n+k}{M}\right\rfloor$. Thus, we have
\begin{equation}
    \label{interp period}
    \begin{aligned}
        &u_k'(t) = u_{k+q}'(t+Mr\e_n) \quad &\text{if } q + k \leq M\\
        &u_k'(t) = u_{k+q-M}'(t+M(r+1)\e_n) \quad &\text{if } q + k > M
    \end{aligned}
\end{equation}

\begin{Proposition}
\label{periodic compactness}
    Consider $n(r) = Mr + q$ for a fixed $q \in\{ 0,\dots, M-1\}$. Consider a sequence $(u^{n(r)})_{r\in \N}$ such that $\sup_n E^1_{n,M}(u^n)\leq C < + \infty$. Then there exists a finite set $S \subset (0,L]$ for which $u^n$ (up to subsequences) converges in $W^{1,\infty}_{\rm loc}(\R \setminus (S+ \Z))$ to a  $\mathbf{u} = (u_1, \dots, u_M)$ as in Theorem \ref{compattezza} such that for any $k=1,\dots,M$ we have
    \[u_k'(t) = u_{q+k}'(t+L)\quad \text{if } q + k \leq M,\qquad u_k'(t) = u_{k+q-M}'(t+L) \quad \text{if } q + k > M\]
    for almost every $t\in S(\mathbf{u}')$.
\end{Proposition}
\begin{proof}
    The convergence is given by Proposition \ref{prop-comp}. What remains to prove is the periodicity property. Suppose at first $q+k\leq M$. Fix $t \notin S(\mathbf{u}')$ and suppose $t+L \notin S(\mathbf{u}')$, then there exists a $\delta>0$ such that $I_\delta(t) = [t-\delta, t+\delta]\subset \R\setminus S(\mathbf{u}')$ and $I_\delta(t+L)\subset \R\setminus S(\mathbf{u}')$.
    For every $k$ fixed, $(u^n_k)'$ converges uniformly on $I_\delta(t)$ and $(u^n_{k+q})'$ on $I_\delta(t+L)$.
    Since $Mr\e_n = L\left(1 - \frac{q}{n}\right)$, we can rewrite \eqref{interp period} as
    \[(u_k^n)'(t) = (u^n_{k+q})'\Big(t+L-\frac{Lq}{n}\Big)\]
Taking the limit in $n$, the left-hand side is converging to $u'_k(t)$, while for the right-hand side we can notice that, since $\mathbf{u}'$ is piecewise constant, $u'_{k+q}(t+L) = u_{k+q}'\Big(t+L-\frac{Lq}{n}\Big)$ for every $n$ large enough, so that
\begin{eqnarray*}
\abs{(u^n_{k+q})'\Big(t+L-\frac{Lq}{n}\Big) - u'_{k+q}(t+L)} &= &\abs{(u^n_{k+q})'\Big(t+L-\frac{Lq}{n}\Big) - u_{k+q}'\Big(t+L-\frac{Lq}{n}\Big)}\\
&\leq& \sup_{s \in I_\delta(t+L)} \abs{(u^n_{k+q})'(s) - u_{k+q}'(s)} \xrightarrow[n\rightarrow +\infty]{} 0. \end{eqnarray*}
    The case $q+k>M$ follows similarly.
\end{proof}

\begin{Remark} Note that at the limit we have the boundary conditions
    $u_k'(0^+) = u_{q+k}'(L^+)$,
    where $q+k$ is intended modulus $M$. This holds also if  $0\in S(\mathbf{u}')$. 
\end{Remark}

\begin{Proposition} \label{identita con le alpha}
    Given $(u^n)_{n\in\N}$ such that $\sup_n E^1_{n,M}(u^n)=: C < + \infty$. Let $0=x_0<x_1<\dots<x_N=L$ and let $\alpha_1,\dots, \alpha_N \in \R$ be such that $\mathbf{M}^{\alpha_j} \subseteq M_{\ell/L}$ and  $u^n$ converges to $\mathbf{u}$ satisfying
    \[\mathbf{u}'(t) = \mathbf{z}^{\alpha_j} \in \mathbf{M}^{\alpha_j}, \quad \text{for } t\in (x_{j-1}, x_j).\]
    Then, $\frac{\ell}{L} = \sum_{j=1}^N \alpha_j(x_j-x_{j-1})$.
\end{Proposition}
\begin{proof}
    Suppose $L=1$. By the boundary conditions we have
\begin{eqnarray*}
\ell= (u^n_n - u^n_{M\lfloor n/M\rfloor}) + \sum_{i=0}^{\lfloor n/M\rfloor} u^n_{M(i+1)}-u^n_{Mi}=  (u^n_n - u^n_{M\lfloor n/M\rfloor}) + \sum_{i=0}^{\lfloor n/M\rfloor} M\e_n\frac{u^n_{M(i+1)}-u^n_{Mi}}{M\e_n}.
\end{eqnarray*}
    Now, fixed $\delta>0$, we define $I_\delta^j = (x_j-\delta, x_j+\delta)\cap[0,1]$, for $j=0,\dots, N$, and we call $S_\delta = \cup _{j=0}^N I_\delta^j$.
    From Proposition \ref{prop-comp} the interpolation derivative $(u^n_k)'(t)$ converges to $z_k^{\alpha_j}$ on $[x_{j-1}+\delta, x_j-\delta]$. Then for an index $i$ such that $[Mi\e_n, M(i+1)\e_n]\subset [x_{j-1}+\delta, x_j-\delta]$, we have
    \[\frac{u^n_{M(i+1)}-u^n_{Mi}}{M\e_n} = \frac{1}{M}\sum_{m=Mi}^{M(i+1)-1} \frac{u^n_{m+1}-u^n_M}{\e_n} = \frac{1}{M}\sum_{k=1}^{M} \frac{u^n_{k,Mi+1}-u^n_{k,Mi}}{\e_n} \xrightarrow[n\rightarrow + \infty]{} \frac{1}{M}\sum_{k=1}^{M} z^{\alpha_j}_k = \alpha_j \]
    Let $k^+_{j,n}$ be the smallest index $k$ such that $M k\e_{n} \geq x_j+\delta$ and similarly let $k^-_{j,n}$ be the largest index $k$ such that $Mk\e_n \leq x_j-\delta$, so that $\lim_n \e_n Mk^-_{j,n} = x_j-\delta$ and $\lim_n \e_n Mk^+_{j,n} = x_j+\delta$. We also define $k^-_{0,n} = 0$ and $k^+_{N,n} = \lfloor \frac{n}{M}\rfloor$. Then,
\begin{eqnarray*}
    \abs{\sum_{i=k^+_{j-1,n}}^{k^-_{j,n}-1} M\e_n\frac{u^n_{M(i+1)}-u^n_{Mi}}{M\e_n}- \alpha_j (x_j-x_{j-1} - 2\delta)}
&\leq& \abs{\sum_{i=k^+_{j-1,n}}^{k^-_{j,n}-1} M\e_n\left(\frac{u^n_{M(i+1)}-u^n_{Mi}}{M\e_n} - \alpha_j\right)}\\
&&\hskip-2cm + \abs{\alpha_j(x_j-\delta - \e_n Mk^-_{j,n})} + \abs{\alpha_j(x_{j-1}+\delta - \e_n Mk^+_{j-1,n})}, 
    \end{eqnarray*}
    which tends to $0$ as $n$ tends to infinity. On the other hand, since $\sup_n E^1_{n,M}(u^n)=C$, there exists a constant $C_1$ independent from $n$ such that for every index $i$ we have
    \[\mathcal{E}_n^i(u^n) \leq C \implies \psi_1\left(\frac{u^n_{i+1}-u^n_i}{\e_n}\right) - r_\ell\left(\frac{u^n_{i+1}-u^n_i}{\e_n}\right)\leq C_1,\]    
    where we used the definition of $\mathcal{E}_n^i$ in \eqref{mathcal E}.
    Thus, from the superlinear growth of $\psi_1$, we must have that there exists an $R>0$ such that 
    $\abs{\frac{u^n_{i+1}-u^n_i}{\e_n}} \leq R$.
    Then, we can conclude
\begin{eqnarray*}
\abs{\ell- \sum_{j=1}^N\sum_{i=k^+_{j-1,n}}^{k^-_{j,n}-1} M\e_n\frac{u^n_{M(i+1)}-u^n_{Mi}}{M\e_n}}& =& \abs{u^n_n - u^n_{Mk^+_{N,n}} + \sum_{j=0}^{N} \sum_{i=k^-_{j,n}}^{k^+_{j,n}-1} M\e_n\frac{u^n_{M(i+1)}-u^n_{Mi}}{M\e_n}} 
\\
 &=&\abs{\sum_{i = Mk^+_{N,n}}^{n-1}\e_n\frac{u^n_{i+1} - u^n_i}{\e_n} + \sum_{j=0}^{N} \sum_{i=Mk^-_{j,n}}^{Mk^+_{j,n}-1} \e_n\frac{u^n_{i+1}-u^n_{i}}{\e_n}}\\
 &\le& R\e_n \bigg(M\sum_{j=0}^{N}(k^+_{j,n} - k^-_{j,n}) + n- Mk^+_{N,n}\bigg).
 \end{eqnarray*}
Sending $n$ to $+\infty$ we have
    \[\abs{\ell- \sum_{j=1}^N \alpha_j (x_j-x_{j-1} - 2\delta)} \leq 2RN\delta\]
    Finally, by the arbitrariness of $\delta>0$ we must have $\ell= \sum_{j=1}^N \alpha_j (x_j-x_{j-1})$.
\end{proof}

\begin{Remark} Let $\Phi=\Phi ^{(\ell)}$ be as in Definition \ref{def-trans}. We note that
the minimum is well defined because of Weierstrass' Theorem. Moreover, the infimum in $N \in \N$ in \eqref{formula magica} can be replaced with the limit as $N \rightarrow + \infty$ because of the decreasing monotonicity.
 The terms inside the sums are $0$ for $i\geq N$ or $i\leq -N-M$ when $\mathbf{z},\mathbf{z}' \in  M_{\ell}$.
 The function 
 $\Phi \colon \mathbf{M}_\ell\times \mathbf{M}_\ell \rightarrow [0, \infty]$ is a subadditive function. In particular the functional defined by the right-hand side of \eqref{Phifu} is lower-semicontinuous.
 \end{Remark} 

\begin{Remark} 
Given $\xi \in \R$, we can consider  
\begin{align}
    \Phi (\mathbf{z}, \mathbf{z}', \xi) : = \inf_{N \in \N} \min_u  \bigg\{ & \sum_{i \in \Z} \bigg ( \psi_M \left ( \frac{u_{i+M}-u_i}{M} \right) 
    + \frac{1}{M} \sum_{k=0}^{M-1} \psi_1 (u_{i+k+1}-u_{i+k})- r _\ell\Big ( \dfrac{u_{i+M}-u_i}{M}  \Big) \bigg ) :\notag \\
    & u : \Z \rightarrow \R,\,  u= u_{\mathbf{z}} + \xi_1 \text{ for } i \leq -N,\, u= u_{\mathbf{z}'} + \xi_2 \text{ for } i \geq N , \xi = \xi_2-\xi_1  \Big \} \\
    = \inf_{N \in \N} \min_u  \bigg\{ &\sum_{i \in \Z} \bigg ( \psi_M \left ( \frac{u_{i+M}-u_i}{M} \right) 
    + \frac{1}{M} \sum_{k=0}^{M-1} \psi_1 (u_{i+k+1}-u_{i+k})- r _\ell\Big ( \dfrac{u_{i+M}-u_i}{M}  \Big) \bigg ) 
    \notag \\
    & u : \Z \rightarrow \R,\,  u= u_{\mathbf{z}} \text{ for } i \leq -N,\, u= u_{\mathbf{z}'} + \xi \text{ for } i \geq N \Big \}.
\end{align} 
The two formulas are equal because  the sums which are involved are invariant under vertical translations of $u$.

Now note that $\Phi (\mathbf{z}, \mathbf{z}', \xi)=\Phi (\mathbf{z}, \mathbf{z}')$ for every $\xi \in \R$, under the assumption that $\psi_M, \psi_1$ are $C^1$ functions in a neighbourhood of $\mathbf{M}_\ell$. Indeed, if $(u^N)_N$ is a minimizing sequence in the definition of $\Phi (\mathbf{z}, \mathbf{z}', \xi)$, then we can consider the following new sequence
\[
\Tilde{u}^{2N} = 
\begin{cases}
    u_{\mathbf{z}'} & \text{ for } i \geq 2N \\
    u_{\mathbf{z}'} + 2\xi - i \dfrac{\xi}{N} & \text{ for } N \leq i \leq 2N \\
    u^N & \text{ for } i \leq N.
\end{cases}
\]
Observe that varying $N \in \N$, all $\Tilde{u}^{2N}$'s are competitors in the definition of $\Phi (\mathbf z, \mathbf z')$. Let $\mathbf{z} \in \mathbf{M}^\alpha$ and  $\mathbf{z}' \in \mathbf{M}^{\alpha'}$ wih $\alpha, \alpha' \in \R$ such that $\mathbf{M}^{\alpha} \subseteq M_{\ell}$ and $\mathbf{M}^{\alpha'} \subseteq M_{\ell}$.
We compute the difference between the involved sums for $\Tilde{u}^{2N}$ and the ones for $u_N$ by
\begin{equation}
    \sum_{i=N}^{2N} \bigg( \psi_M \Big( \alpha' - \dfrac{\xi}{N} \Big) + \dfrac{1}{M} \sum_{k=0}^{M-1} \psi_1 \Big( z_k' - \dfrac{\xi}{N} \Big) - r_\ell \Big ( \alpha' - \dfrac{\xi}{N} \Big) \bigg) + o \Big( \frac{1}{N} \Big).
\end{equation}
Now observe that the $C^1$ function 
$
f(\eta) = \psi_M \left ( \alpha' - \eta \right ) + \dfrac{1}{M} \sum_{k=0}^{M-1} \psi_1 \left ( z_k' - \eta \right ) - r _\ell( \alpha' - \eta )
$
is always non-negative and it is equal to $0$ if $\eta=0$, so that $\eta=0$ is a minimum point for $f$. Hence by Fermat's Theorem $f'(0)=0$ and so 
\[
\lim _{N \rightarrow + \infty} \sum_{i=N}^{2N} \bigg ( \psi_M \Big( \alpha' - \dfrac{\xi}{N} \Big) + \dfrac{1}{M} \sum_{k=0}^{M-1} \psi_1 \Big( z_k' - \dfrac{\xi}{N} \Big) - r _\ell\Big ( \alpha' - \dfrac{\xi}{N} \Big) \bigg) =0.
\]
In particular, we proved that $\Phi(\mathbf{z}, \mathbf{z'}) \leq \Phi(\mathbf{z}, \mathbf{z'}, \xi)$. Arguing in the same way one can show the reverse inequality.
\end{Remark}
\begin{Remark}
\label{translation stability}
    A last useful remark is that $\Phi(\mathbf{z}, \mathbf{z}')$ is stable under cyclic permutations of the entries of $\mathbf{z}$ and $\mathbf{z}'$. More precisely, if we define for $q=1,\dots, M-1$ the cyclic permutation $\sigma_q : \mathbf{M}_{\ell} \rightarrow \mathbf{M}_{\ell}$ given by
    \begin{equation}
    \label{sigma_q}
    \sigma_q : \mathbf{z} = (z_1,\dots, z_M) \longmapsto \sigma_q (\mathbf{z}) := (z_{M+1-q},\dots, z_M, z_1, \dots, z_{M-q}),
    \end{equation}
    then $\Phi(\sigma_q(\mathbf{z}), \sigma_q(\mathbf{z}')) = \Phi (\mathbf{z}, \mathbf{z}')$. 
    Indeed, if $u$ is a competitor for the minimum problem in $\Phi(\mathbf{z}, \mathbf{z}')$, then the translation $(T_qu)_i=u_{i-q}$ preserves the value of the objective function in the definition of $\Phi$ and it satisfies
    \[T_qu= u_{\sigma_q(\mathbf{z})} \text{ for } i \leq -N-q, \quad \quad T_qu= u_{\sigma_q(\mathbf{z}')} \text{ for } i \geq N-q.\]
    Thus, $T_qu$ is a competitor for the minimum problem in $\Phi(\sigma_q(\mathbf{z}), \sigma_q(\mathbf{z}'))$. Clearly we can argue in the same way with the inverse translation, proving the claim.
\end{Remark}

\begin{theorem}[first-order $\Gamma$-limit with periodic boundary conditions]\label{1-teo-q}
    Assume that $\psi_M$ is a non-degenerate strictly convex $C^1$ function, $\psi_1$ is of class $C^1$ in a neighbourhood of $\mathbf{M}_{\ell/L}$ and $\psi_0^{**}$ is differentiable in $\frac{\ell}{L}$. Constraining $n \equiv q\; \text{mod }M$, we define the domain 
    \begin{align*}
        D_q^\# = \{\mathbf{u}: \R\rightarrow \R^M \,|\,& \mathbf{u} \text{ piecewise affine},\, \mathbf{u}'(t)\in \mathbf{M}_{\ell/L}\; \hbox{ for all } t\in \R\setminus S(\mathbf{u}'),\\ &\#(S(\mathbf{u}')\cap (0,L]) <+\infty, \, \sum_{k=1}^{M} u_k(t)- M\frac{\ell}{L} t \; \text{ L-periodic,} \\ &u'_k(t) = u'_{[q + k]_{\text{mod }M}}(t+L) \text{ for } k=0,\dots,M-1\}.
    \end{align*}
 Then 
    $E^1_{n,M}$ $\Gamma$-converges, with respect to $L^{\infty}$ topology, to 
    \[E^1_q (\mathbf{u}) =  \sum_{t\in S(\mathbf{u}')\cap (0,1]}\Phi^{(\ell/L)}(\mathbf{u}'(t-), \mathbf{u}'(t+)),\]
    with the domain dom$(E_q^{1,\#}) = D^\#_q$.
\end{theorem}

\begin{proof}
We  first  consider the case in which $L=1$ and $n \equiv 0$ mod $M$.

 {\em Lower bound}. Let $u^n \rightarrow u$ in $L^{\infty}(0,1)$ be such that $E^{1,\ell}_{n,M}(u^n) \leq C < +\infty$ for every $n \in \N$. Then, by Proposition \ref{periodic compactness}, there exists a finite set $
    S: = \{ x_1, \dots x_N \} \subset (0,1]
$,
with $0< x_1 < \dots x_{N-1} < x_N\leq 1$, and there exist $\mathbf{z}_1, \dots, \mathbf{z}_N \in  M_{\ell}$ such that $u^n$ (up to subsequences) converges in $W^{1,\infty}_{\rm loc}(\R \setminus (S+\Z))$ to a M-tuple $\mathbf{u} \in D^\#_0$ such that 
$
    \mathbf{u}'(t) = \mathbf{z}_j \in \mathbf{M}^{\alpha_j}$ for $t\in (x_{j-1}+k, x_j+k)$ and for all $k \in \Z$ for $j\in\{1,\dots,N\}$,
    with $x_0 = x_N-1$. For $j \in \{1,...,N \}$, let $(k_n^j)_n$ be a sequence of natural numbers divisible by $M$ such that 
    \begin{equation} \label{definizione k}
   \lim_{n \rightarrow +\infty} k_n^j\e_n - x_j=0.
    \end{equation}
    Moreover, let $(h^j_n)_n$ be a sequence in $M\N$ such that 
    \begin{equation}
        \lim_{n \rightarrow + \infty} \e_n h_n^j = \frac{x_j+x_{j-1}}{2}.
    \end{equation}
    Now, we can write
    \begin{align}
        E^1_{n,M}(u^n) 
        = \sum_{j=1}^{N} \sum_{i = h_n^{j} } ^{h_n^{j+1} -1 } \left ( \psi_M\left(\frac{u^n_{i+M}-u^n_{i}}{M\e_n}\right) + \frac{1}{M}\sum_{s=0}^{M-1}\psi_1\left(\frac{u^n_{i+s+1}-u^n_{i+s}}{\e_n}\right) -r\left(\frac{u^n_{i+M}-u^n_{i}}{M\e_n}\right) \right ) \label{riscrittura E^1},
    \end{align}
    where $r=r_\ell$, and we have used the notation $h_n^{N+1} = h^1_n + n$ and that by periodicity we can choose the endpoints $h_n^1$ and $h_n^1+n$ without changing the sum.    
    In order to recover $\Phi$, we define 
    \begin{equation}
    \label{u tilde}
    \Tilde{u}^n_i = \left\{
    \begin{aligned}
        & u_{\mathbf{z}_j}(i) - u_{\mathbf{z}_{j+1}}(h^j_n-k^j_n) +\frac{u^n_{h^j_n}}{\e_n}\quad &\text{for } i\leq h^j_n-k^j_n\\
        & \frac{u^n_{i+k^j_n}}{\e_n}\quad &\text{for } h^j_n-k^j_n\leq i\leq h^{j+1}_n-k^j_n\\
        & u_{\mathbf{z}_{j+1}}(i) - u_{\mathbf{z}_{j+1}}(h^{j+1}_n-k^j_n)+ \frac{u^n_{h^{j+1}_n}}{\e_n}\quad &\text{for } i\geq h^{j+1}_n-k^j_n.
    \end{aligned}\right.
    \end{equation}
    Since $\mathbf{z}_j\in\mathbf{M}^{\alpha_j}$, we note that for $i\geq h^{j+1}_n-k^j_n$, we have
    \[
    \psi_M\left(\frac{\Tilde{u}^n_{i+M}-\Tilde{u}^n_{i}}{M}\right) + \frac{1}{M}\sum_{s=0}^{M-1}\psi_1\left(\Tilde{u}^n_{i+s+1}-\Tilde{u}^n_{i+s}\right) = \psi_0(\alpha_{j+1}) = r(\alpha_{j+1}) = r\left(\frac{u^n_{i+M}-u^n_{i}}{M\e_n}\right)
    \]
    and similarly for $i\leq h^{j}_n-k^j_n-M$. For $h^j_n-k^j_n\leq i\leq h^{j+1}_n-k^j_n -M$, instead, we have
    \[
    \Tilde{u}^n_{i+1}-\Tilde{u}^n_{i} = \frac{u^n_{i+k^j_n+1}-u^n_{i+k^j_n}}{\e_n} \quad \text{and} \quad \frac{\Tilde{u}^n_{i+M}-\Tilde{u}^n_{i}}{M} = \frac{u^n_{i+k^j_n+M}-u^n_{i+k^j_n}}{M\e_n}.
    \]
    Then, defining $\xi^n_j = u_{\mathbf{z}_{j+1}}(h^{j+1}_n-k^j_n) - u_{\mathbf{z}_{j+1}}(h^j_n-k^j_n) -\frac{u^n_{h^{j+1}_n} - u^n_{h^j_n}}{\e_n}$, we can rewrite 
    \begin{align*}
    &\sum_{i = h_n^{j} } ^{h_n^{j+1} -1 } \left ( \psi_M\left(\frac{u^n_{i+M}-u^n_{i}}{M\e_n}\right) + \frac{1}{M}\sum_{s=0}^{M-1}\psi_1\left(\frac{u^n_{i+s+1}-u^n_{i+s}}{\e_n}\right) - r\left(\frac{u^n_{i+M}-u^n_{i}}{M\e_n}\right)\right ) \\
    =&\  \sum_{i \in \Z} \Bigg ( \psi_M \left ( \frac{\Tilde{u}^n_{i+M}-\Tilde{u}^n_i}{M} \right) 
    + \frac{1}{M} \sum_{k=0}^{M-1} \psi_1 (\Tilde{u}_{i+k+1}^n-\Tilde{u}^n_{i+k}) - r\left(\frac{\Tilde{u}^n_{i+M}-\Tilde{u}^n_{i}}{M}\right) \Bigg )+ \omega_n \\
    \geq &\ \Phi\left(\mathbf{z}_j, \,\mathbf{z}_{j+1}, \,\xi_j^n\right) + \omega_n = \Phi\left(\mathbf{z}_j, \,\mathbf{z}_{j+1}\right) + \omega_n.
    \end{align*}
 The error  $\omega_n$ comes from the difference in behaviour between $\Tilde{u}^n_i$ and $u^n_{i+k^j_n}$ when $h^{j+1}_n-k^j_n -M < i < h^{j+1}_n-k^j_n$ or $h^{j}_n-k^j_n -M < i < h^{j}_n-k^j_n$; more precisely, $\omega_n$ is the sum, for those $i$, of terms of the kind 
     \[
     \begin{cases}
     \psi_M \bigg ( \dfrac{u^n_{i+ k^j_n+M} - u^n_{i+k^j_n}}{M \e_n } \bigg) - \psi_M \left ( \dfrac{\Tilde{u}^n_{i+M}-\Tilde{u}^n_i}{M} \right) & \text{for } h^{j+1}_n-k^j_n -M < i < h^{j+1}_n-k^j_n, \\
     - \psi_M \left ( \dfrac{\Tilde{u}^n_{i+M}-\Tilde{u}^n_i}{M} \right) & \text{for } h^{j}_n-k^j_n -M < i < h^{j}_n-k^j_n.
     \end{cases}
     \]
     and similar terms for $\psi_1$. To show that $\lim_n \omega_n= 0$, from the continuity of $\psi_M$, $\psi_1$ and $r$, it is sufficient to show that 
    \[
    \begin{cases}
        \lim_n \abs{\dfrac{u^n_{i + k^j_n+1}-u^n_{i+k^j_n}}{\e_n} - u_{\mathbf{z}_{j+1}}(i+1)-u_{\mathbf{z}_{j+1}}(i)}=0, & \text{for } i\in (h^{j+1}_n - k^j_n -M, h^{j+1}_n - k^j_n +M)\\
        \lim_n \abs{\dfrac{u^n_{i + k^j_n+1}-u^n_{i+k^j_n}}{\e_n} - u_{\mathbf{z}_{j}}(i+1)-u_{\mathbf{z}{j}}(i)}=0, & \text{for } i\in (h^{j}_n - k^j_n -M, h^{j}_n - k^j_n +M)
    \end{cases}.
    \]
    Let $i_M$ denote the residual  class of $i$ with respect to $M$. Then, by the definition of $u_{\mathbf{z}_{j+1}}$ as in \eqref{defn u_z} we have
    \[
    u_{\mathbf{z}_{j+1}}(i+1) - u_{\mathbf{z}_{j+1}}(i) = (\mathbf{z}_{j+1})_{i_M+1}.
    \]
    On the other hand, since there exists a compact around the midpoint of $(x_j,x_{j+1})$ which contains $\e_n(h^{j+1}_n-M, h^{j+1}_n+M)$ for all $n$, we can use the locally uniform convergence of $\mathbf{u}_n'$ to $\mathbf{u}'$ to gain
    \[
    \dfrac{u_{i+1+k_n^j}^n- u^n_{i+1+k_n^j}}{\e_n} = (\mathbf{u}^n)'_{i_M+1}(\e_n(i+k_n^j)) \longrightarrow \mathbf{u}'_{i_M+1}\left(\frac{x_{j+1}+x_j}{2}\right) =  (\mathbf{z}_{j+1})_{i_M+1} \qquad \text{for } n\rightarrow + \infty.
    \]
    The case $i\in (h^{j}_n - k^j_n -M, h^{j}_n - k^j_n +M)$ is analogous. This proves the claim, and consequently that $\lim_n \omega_n = 0$. 
    In particular, putting everything together we conclude
    \[
    \liminf_n E^1_{n,M}(u^n) \geq \sum_{j=1}^N \Phi\left(\mathbf{z}_j, \,\mathbf{z}_{j+1}\right).
    \]
    
    In the case $q\neq 0$, the only thing that changes is that we cannot define $h^{N+1}_n=h^1_n + n$, because in this way it will not be divisible by $M$, thus, we take $h^{N+1}_n$ the highest index multiple of $M$ below $h^1_n+n$. Consequently, in the decomposition \eqref{riscrittura E^1} there will appear $q$ residual terms of the kind:
    \[\psi_M\left(\frac{u^n_{i+M}-u^n_{i}}{M\e_n}\right) + \frac{1}{M}\sum_{s=0}^{M-1}\psi_1\left(\frac{u^n_{i+s+1}-u^n_{i+s}}{\e_n}\right) -r\left(\frac{u^n_{i+M}-u^n_{i}}{M\e_n}\right),\]
    but since $i\e_n$ is close to $1+\frac{x_1+x_0}{2}$; that is, far from any critical point $x_j$, by continuity we have that they tend to $0$.

    {\em Upper bound.} Let $\mathbf{u}$ be such that $E^1_0( \mathbf{u} ) < \infty$ and suppose , without loss of generality, that $\mathbf{u}(0) = 0$, then by definition there exist $N \in \N$, $\mathbf{z}_1, \dots \mathbf{z}_N \in \mathbf{M}_{\ell}$ and we can write $
    S({\mathbf u}')\cap (0,1]= \{ x_1, \dots x_N \} $, with $0< x_1 < \dots x_{N-1} < x_N\leq 1$, 
    $
    \mathbf{u}'(t) = \mathbf{z}_j \in \mathbf{M}^{\alpha_j}$ for $t\in (x_{j-1}+k, x_j+k)$ and for all $k \in \Z$,
    with $x_0 = x_N-1$. {Up to translations, which do not change the energy, we can always suppose that $1\notin S(\mathbf{u}')$; that is, $x_N<1$.} Moreover, we have that
    \[
    E^1_0( \mathbf u ) = \sum_{j=1}^N \Phi (\mathbf z_{j-1}, \mathbf z_{j}).
    \]
    
    We first consider the case $q= 0$.
    With fixed $\e >0$, there exists $\Tilde{N} = \Tilde{N}(\e) \in \N$ multiple of $M$ and $N$ discrete functions $v^j$ such that
    \begin{eqnarray}
     \label{ineq for limsup}\nonumber
        \Phi ( \mathbf z_{j-1}, \mathbf z_j) &\leq& \sum_{ i \in \Z} \psi_M \Big( \frac{v^j_{i+M}-v^j_i}{M} \Big) 
        + \frac{1}{M} \sum_{k=0}^{M-1} \psi_1 (v^j_{i+k+1}-v^j_{i+k}) - r_\ell \Big( \dfrac{v^j_{i+M}-v^j_i}{M}\Big)\\
        &\leq& \Phi ( \mathbf z_{j-1}, \mathbf z_j) + \e
    \end{eqnarray}
    and
    \[
    v^j = \begin{cases}
        u_{\mathbf z_{j-1}} & \text{ for } i \leq -\Tilde{N} \\
        u_{\mathbf z_{j}} & \text{ for } i \geq \Tilde{N}
    \end{cases}
    \]
    for every $j \in\{ 1 , \ldots N\}$. Consider now the sequence of functions $(u^n)_n$ defined as follows
    \begin{equation}
    \label{rec seq}
    u^n _i = \begin{cases}
        \e_n v^1_{i-k^n_N+n} - \e_n v^1_{n-k^n_N} & \text{ for } 0 \leq i \leq k^n_1 - \Tilde{N} \\
        \e_n v^{2}_{ i - k^{n}_1} + \e_n D_2 & \text{ for } k^n_{1} - \Tilde{N} \leq i \leq k^n_{2} - \Tilde{N} \\    
        \dots \\
        \e_n v^{N}_{ i - k^{N-1}_n} + \e_n D_{N} & \text{ for } k^n_{N-1} - \Tilde{N} \leq i \leq k^n_{N} - \Tilde{N}\\
        \e_n v^1_{i-k^N_n} + \e_n D_1 & \text { for }  k^n_N - \Tilde{N} \leq i \leq n,
    \end{cases}
    \end{equation}
    where $k^n_j : = \min \{ k \in \N : k \geq x_j n$ and $k$ is multiple of $M \}$ and
    \begin{align*}
    & D_2 = v^1_{k_1^n-k_N^n+n- \Tilde{N}} -v^2_{- \Tilde{N} } - v^1_{n-k^n_N}\\
    & D_3 = D_2 + v^2_{k^n_2- k^n_1 - \Tilde{N} } - v^3_{-\Tilde{N}} \\
    & \dots \\
    & D_{N} = D_{N-1} + v^{N-1}_{k^n_{N-1}-k^n_{N-2}- \Tilde{N}} - v^N_{-\Tilde{N}} \\
    & D_1 = D_N + v^{N}_{k^n_{N}-k^n_{N-1}- \Tilde{N}} - v^1_{-\Tilde{N}}.
    \end{align*}
     We note that $u^n_n - u^n_0 = \sum_{j=1}^N \alpha_j (k_j^n-k^n_{j-1})\e_n$, which in general is different from $\ell$; therefore, in order to adjust the periodicity conditions, we apply a linear correction term
     $
     \Tilde{u}^n_i = u^n_i + \delta_n \frac{i}{n}$,
     where $\delta_n = \ell - \sum_{j=1}^N \alpha_j (k_j^n-k^n_{j-1})\e_n$. Since now the the boundary condition is satisfied, we can extend $\Tilde{u}^n$ to $\R$ such that $\Tilde{u}^n(t) - \ell t$ is $1$-periodic.
     From Proposition \ref{identita con le alpha} and the fact that $|n x_j - k^n_j| \leq M$ we have
   \begin{equation}
     \label{crescita delta_n} \delta_n = \sum_{j=1}^N \alpha_j (x_j-x_{j-1} - k^n_j \e_n + k^n_{j-1}\e_n) = O(\e_n)
         \implies \exists C>0 : \;\ (k^n_j-k^n_{j-1})\delta_n \leq C\  \forall j, n.
     \end{equation}
     Since $\mathbf{u}^n$ is converging uniformly to $\mathbf{u}$
     and $\delta_nx \rightarrow 0$ uniformly in $x\in [0,1]$, we have that
     $\Tilde{u}^n \rightarrow \mathbf{u} $ in  $L^\infty([0,1])$ in the sense of Definition \ref{def conv}.
     As for the convergence of the energy, if we let
     \[\mathcal{E}_i^n(u) = \psi_M\left(\frac{u_{i+M}-u_i}{M\e_n}\right) + \frac{1}{M}\sum_{k=0}^{M-1}\psi_1\left(\frac{u_{i+k+1}-u_{i+k}}{\e_n}\right) - r_\ell\left(\frac{u_{i+M}-u_i}{M\e_n} \right),\]
     then, by \eqref{ineq for limsup} we have
     \[E^1_0(\mathbf{u}) + \e \geq E^1_{n,M}(\Tilde{u}^n) + \sum_{j=1}^N \Bigg(R_n^j - \sum_{i=k^n_{j-1}+\Tilde{N}}^{k^n_j - \Tilde{N} - M} \mathcal{E}^n_i(\Tilde{u}^n)\Bigg),\]
     where $R_n^j = \sum_{i=-\Tilde{N}-M+1}^{\Tilde{N}-1} \mathcal{E}^n_i (\e_n v^j)- \mathcal{E}^n_{i + k^n_{j-1}}(\Tilde{u}^n)$. 
     However, this $R_n^j$ involves only a finite number of indices, independently from $n$; thus, since $\frac{\Tilde{u}^n_{i+M}-\Tilde{u}^n_i}{M\e_n} = \frac{v^j_{i+M}-v^j_i}{M} + \delta_n$ and $\frac{\Tilde{u}^n_{i+1}-\Tilde{u}^n_i}{\e_n} = v^j_{i+1}-v^j_i + \delta_n$, by continuity of $\psi_1, \psi_M$ and $r_\ell$ we have
     \[\lim_n \sum_{j=1}^N R^j_n = 0.\]
     Instead, for $i \in\{ k^n_{j-1}-\Tilde{N}, \dots, k^n_j -\Tilde{N} -M\}$ we have $v^j_{i-k^n_{j-1}} = u_{\mathbf{z}_j}(i-k^n_{j-1})$ and then
     \[\mathcal{E}^n_i(\Tilde{u}^n) = \psi_M\left(\alpha_j + \delta_n\right) + \frac{1}{M}\sum_{k=1}^{M}\psi_1\left((\mathbf{z}_j)_k+ \delta_n\right) - r_\ell\left(\alpha_j +\delta_n \right)=: f_j(\delta_n).\]
     Note that we lost the dependence on $i$, so that, by \eqref{crescita delta_n}, we have
     \[\sum_{i=k^n_{j-1}+\Tilde{N}}^{k^n_j - \Tilde{N} - M} \mathcal{E}^n_i(\Tilde{u}^n) = (k^n_j - k^n_{j-1} - M) f_j(\delta_n) = (k^n_j - k^n_{j-1} - M)\delta_n \left(\frac{f_j(\delta_n)}{\delta_n}\right) \leq C\frac{f_j(\delta_n)}{\delta_n}.\]
     By Fermat's Theorem we can conclude that $\lim_n \frac{f_j(\delta_n)}{\delta_n} = f'(0) = 0$; in conclusion
     \[E^1_0(\mathbf{u}) + \e \geq \limsup_n E^1_{n,M}(\Tilde{u}^n),\]
     and the claim is proved by the arbitrariness of $\e$. 
     
     When $q\neq 0$, the proof proceeds similarly, but we have to pay attention to the boundary mismatch in the cycle of slopes in $\Tilde{u}$. This can be fixed by taking into account that, by Remark \ref{translation stability}, we can write the energy as
 \[E^1_q(\mathbf{u}) = \sum_{j=1}^N \Phi(\sigma_q(\mathbf{z}_{j-1}), \sigma_q(\mathbf{z}_j) ),\]
 where $\sigma_q$ is defined as in \eqref{sigma_q}. Then, we can find some $\Tilde{N}$ multiple of $M$ and some $v_j$ as in \eqref{ineq for limsup} with the condition
 \[
 v^j = \begin{cases}
     u_{\sigma_q(\mathbf{z}_{j-1})} & \text{for } i \leq -\Tilde{N}\\
     u_{\sigma_q(\mathbf{z}_{j})} & \text{for } i \geq \Tilde{N}.
 \end{cases}
 \]
 Hence, we define $k^n_j : = \min \{ k \in \N : k \geq x_j n$ and $k\equiv q$ mod $M \}$ and we choose $u^n$ as in \eqref{rec seq}. In this way, the interpolations $\mathbf{u}^n$ are still converging uniformly to $\mathbf{u}$ because
 \[
 u_{\sigma_q(\mathbf{z}_{j})}(i+k_j^n+1) - u_{\sigma_q(\mathbf{z}_{j})}(i+k_j^n) = u_{\sigma_q(\mathbf{z}_{j})}(i+q+1) - u_{\sigma_q(\mathbf{z}_{j})}(i+q) = u_{\mathbf{z}_{j}}(i+1) - u_{\mathbf{z}_{j}}(i).
 \]
 Since $n-k_N^n$ is divisible by $M$ and $\frac{u^n_n-u^n_{n-1}}{\e_n} = \sigma_q(\mathbf{z_1})_M$, we can extend $u^n$ by periodicity on the whole $\R$ defining a recovery sequence.
\end{proof}


\bibliography{bibliography.bib}
\bibliographystyle{abbrv}

\end{document}